\documentclass[twoside,11pt]{amsart}

\usepackage[utf8]{inputenc}

\usepackage{hyperref}
\hypersetup{
  colorlinks   = true,    
  urlcolor     = blue,    
  linkcolor    = blue,    
  citecolor    = red      
}
\usepackage{mathtools}
\usepackage{latexsym}

\usepackage{color}

\usepackage{amsmath}
\usepackage{amsfonts}
\usepackage{amssymb}
\usepackage{graphicx}
\usepackage{color}
\usepackage[normalem]{ulem}
\usepackage[left=2cm,right=2cm,top=2cm,bottom=2cm]{geometry}
\newtheorem{theorem}{Theorem}[section]
\newtheorem{lemma}{Lemma}[section]

\newtheorem{prop}{Proposition}[section]

\def\bhag#1{\noindent
\setcounter{equation}{0}
\section{#1}
}

\def\RR{{\mathbb R}}
\def\CC{{\mathbb C}}
\def\ZZ{{\mathbb Z}}

\def\PPI{{{\rm I}\kern-1pt\Pi}}
\def\SS{{\mathbb S}}

\def\b #1;{{\bf #1}}
\def\x{{\bf x}}

\def\y{{\bf y}}

\def\w{{\bf w}}

\def\WW{{\mathbb W}}

\def\be{\begin{equation}}
\def\ee{\end{equation}}
\def\bea{\begin{eqnarray}}
\def\eea{\end{eqnarray}}
\def\ls{\lesssim}

\def\donchitre#1#2{\vskip 6.5cm\noindent
\parbox[t]{1in}{\special{eps:#1.eps x=6.5cm y=5.5cm}}
\hbox to 7cm{}\parbox[t]{0.0cm}{\special{eps:#2.eps x=6.5cm y=5.5cm}}}

\def\tn{|\!|\!|}
\def\XX{{\mathbb X}}

\def\bs#1{{\boldsymbol{#1}}}
\def\bsw{\bs{\omega}}

\newcommand\ol[1]{\overline{#1}}

\newcommand{\ben}{\begin{enumerate}}
\newcommand{\een}{\end{enumerate}}
\newcommand{\bit}{\begin{itemize}}
\newcommand{\eit}{\end{itemize}}

\newcommand{\bS}{\mathbb{S}}
\newcommand{\R}{\mathbb{R}}
\newcommand{\bsx}{\boldsymbol{x}}
\newcommand{\bsy}{\boldsymbol{y}}

\newcommand{\Kex}{\mathbb{K}}  

\numberwithin{equation}{section}

\title{Spectral Bounds for Kernel Quadrature}

\author{A. Cloninger}
\address{Department of Mathematics and
Halicioglu Data Science Institute
University of California, San Diego}
\email{acloninger@ucsd.edu}
\author{Q.~T.~Le Gia}
\address{School of Mathematics and Statistics, University of New South Wales, Sydney, Australia.}
\email{qlegia@unsw.edu.au}
\author{H. N. Mhaskar}
\address{Institute of Mathematical Sciences, Claremont Graduate University, Claremont, CA 91711, U.S.A.}
\email{hrushikesh.mhaskar@cgu.edu}
\thanks{The research of HNM was supported in part by ONR grants N00014-23-1-2394, N00014-23-1-2790, and in part, by a visiting researcher position at Sydney Mathematical Research Institute (SMRI), Australia. The research of AC was supported in part by NSF CISE-2403452.} 
\begin{document}
\begin{abstract}
A bottleneck in the theory of kernel methods in machine learning is the storage requirement. 
To ameliorate this, a standard trick is to replace the kernel with an explicit feature map. 
Perhaps, the most well known example is the Gaussian kernel which can be expressed in terms of the Fourier features.
Analytically, the kernel $K$ can be expressed in terms of an integral expression that involves a possibly asymmetric kernel $G$ representing the feature map.
Numerically, one needs to approximate this integral by a suitable numerical integration scheme, typically Monte Carlo.
In this paper, we demonstrate that the eigenvalues of $K$ are approximated much better by the eigenvalues of the kernel obtained by discretizing the integral using suitable quadrature formulas instead. 
We illustrate this fact in the case of the Gaussian kernel and neural tangent kernels on the unit sphere of a four dimensional Euclidean space corresponding to the sigmoid and ReLU activation functions.
\end{abstract}

\maketitle

\section{Introduction}
Kernel methods are among the most reliable tools for learning nonlinear functions in
high-dimensional spaces, offering both strong empirical performance and a mature
theoretical foundation. Their central obstacle is scale: a method that works with a
kernel $K$ on a dataset of $M$ points must, in its naive form, construct and operate on the
$M \times M$ Gram matrix $\mathbb{K} = [K(x_k, x_j)]$. Storing this matrix costs $O(M^2)$ memory,
and the linear-algebraic operations at the heart of kernel ridge regression, Gaussian
processes, and kernel SVMs, cost up to
$O(M^3)$ time. For modern datasets this is prohibitive.
 
The dominant response to this bottleneck is to \emph{replace the kernel with an explicit
feature map}. Most kernels of interest admit an integral representation
\begin{equation}
  K(x, y) = \int_{\mathbb{W}}G(w, x)\, \ol{G(w, y)}\, d\mu^*(w), \qquad x, y \in \mathbb{X},
\end{equation}
where $\mu^*$ is a probability measure on a parameter space $\mathbb{W}$ and $G(w, \cdot)$ is a
feature function. If we can approximate this integral by a finite sum over points
$w_1, \dots, w_N$, then mapping each data point $x$ to the explicit feature vector
$z(x) = [G(w_i, x)]_{i=1}^N$ turns kernel evaluation back into an ordinary inner product,
$K(x,y) \approx \langle z(x), z(y)\rangle$. One then runs \emph{linear} methods on the
$N$-dimensional features, with cost linear rather than quadratic in $M$. The quality of
the kernel machine is therefore inherited entirely from the quality of the underlying
numerical integration rule.
 
The most influential instance of this idea is \textbf{Random Fourier Features (RFF)} of
Rahimi and Recht~\cite{rahimi2007random}, which approximates the integral by Monte Carlo:
draw $w_1, \dots, w_N$ i.i.d.\ from $\mu^*$ and form
$\widehat{K}(x,y) = \tfrac{1}{N}\sum_i G(w_i, x) \ol{G(w_i, y)}$. This gives a \emph{pointwise}
approximation with the familiar Monte Carlo rate,
$|K(x,y) - \widehat{K}(x,y)| = O(1/\sqrt{N})$ for any fixed pair $x, y$.  The trouble is that pointwise accuracy is not what kernel methods actually need. Learning
guarantees, conditioning of the regularized system, and generalization bounds all depend
on a \emph{spectral} notion of approximation, meaning that the entire eigenvalue spectrum of
$\widehat{K}$ tracks that of $K$, in the operator sense
\begin{equation}\label{eq:goaleqn}
  (1-\Delta)(K + \lambda I ) \ll (\widehat{K}+\lambda I) \ll (1+\Delta) (K + \lambda I),
\end{equation}
where for symmetric matrices $A$ and $B$, $A\ll B$ means that $B-A$ is positive semi-definite.   
Achieving this with random features is far more demanding than achieving pointwise
accuracy. In practice the number of random measurements $N$ required for a spectral
guarantee often \emph{exceeds the number of data points $M$}, at which point storing and
computing with $[G(w_i, x)]$ is no cheaper than working with the original Gram
matrix. The root cause is statistical: the Monte Carlo
error of a random rule decays only as $N^{-1/2}$, regardless of how smooth or
well-behaved the feature functions are. Randomness cannot exploit structure that a
carefully placed set of points could.
 
This paper treats the integral representation as a problem in \textbf{numerical
quadrature}. Rather than sampling $\mu^*$ at random, we approximate it by a deliberately
constructed, finitely supported quadrature measure $\nu$ that integrates a chosen
subspace of functions exactly. Our central object is the approximation
\begin{equation}
  K(\nu; x, y) = \int_\mathbb{W} G(w, x)\, \ol{G(w, y)}\, d\nu(w),
\end{equation}
and our central question is how the \emph{spectral} error between $K$ and
$K(\nu; \cdot, \cdot)$ depends on the quadrature rule. The payoff of this viewpoint is
twofold. First, a good quadrature rule converges at a rate governed by the
approximation-theoretic quality of the features, which can be exponentially fast for analytic
features, rather than the universal $N^{-1/2}$ ceiling of Monte Carlo. Second, and
crucially for the storage problem above, exact quadrature rules on structured domains can
attain a target accuracy with a \emph{small} number of points $N$, so the explicit
feature vectors $z(x)$ are both short and highly expressive. Fewer, better-placed points
produce more spectral fidelity per unit of storage, which is precisely the resource that
random features waste.
 
\subsection{Related work}
 
\paragraph{ \bf Random features.}
The starting point is the random features framework of Rahimi and
Recht~\cite{rahimi2007random}, which made explicit feature maps a practical alternative to
the kernel trick and established the $O(1/\sqrt{N})$ pointwise rate. Subsequent analyses
sharpened the picture in the direction that matters for learning. Avron, et al~\cite{avron2017random} studied random Fourier features for
kernel ridge regression specifically through the lens of \emph{spectral} matrix
approximation, giving tight bounds on the number of features needed for a spectral
guarantee of the form $K - \lambda I \ll \widehat{K} \ll K + \lambda I$ and showing
how such a guarantee implies statistical bounds for kernel ridge regression. They also showed
that vanilla RFF is suboptimal, and that sampling from a modified, leverage-score
distribution improves the feature count. Our work shares their spectral target, but differs in mechanism: where they
reduce the feature count by \emph{reweighting the sampling distribution}, we eliminate
sampling altogether and place points by quadrature, trading a statistical rate for an
approximation-theoretic one.
 
\paragraph{\bf Quadrature and random features.}
That the two views are connected is well known. Bach~\cite{bach2017equivalence} proved an
equivalence between kernel-based quadrature rules and random feature expansions, with
sample-complexity bounds expressed through the eigenvalues of the integral operator. This
work frames the deep correspondence; our contribution is to push the \emph{quadrature}
side of it concretely, constructing rules on the sphere and on $\mathbb{R}^q$ and tracking
the resulting spectral error of the Gram matrix directly.
 
\paragraph{\bf Deterministic feature maps.}
Closest in spirit is the work of Dao, De Sa, and Ré~\cite{dao2017gaussian}, who construct
deterministic feature maps by approximating the kernel's frequency-domain integral with
Gaussian quadrature, and prove that error $\epsilon$ can be reached with far fewer samples
than the $O(\epsilon^{-2})$ that randomization demands. We differ from their work in two
respects that we regard as the main novelties here. First, they analyze \emph{pointwise /
functional} kernel approximation, whereas our bounds are stated directly in terms of the
\emph{spectrum} of the Gram matrix, and we bound the eigenvalue discrepancy $(\sum_k (\lambda_k - \widehat\lambda_k)^2)^{1/2}$
explicitly. Second, their constructions and their strongest guarantees lean on smoothness
of the kernel and its features; a substantial part of our contribution is to show that the
quadrature advantage persists even when the features are \emph{non-smooth}, including
arc-cosine / ReLU-type kernels~\cite{cho2009kernel} and neural-tangent-style kernels built
from Leaky ReLU and standard ReLU activations, where one would naively expect quadrature
to lose its edge.
 
\paragraph{\bf Memory-constrained approximation.}
Finally, the storage cost of explicit features is itself an active concern. Zhang, May,
Dao, and Ré~\cite{zhang2019lowprecision} address it from the \emph{precision} side,
quantizing random Fourier features to build a high-rank approximation under a fixed memory
budget. We address the same budget from the \emph{count} side: by using a structured
quadrature grid whose points are a Cartesian product across dimensions, we show the Fourier
features admit a tensorized storage scheme whose cost grows \emph{linearly} rather than
exponentially in the ambient dimension. The two approaches are complementary; one shrinks
the bits per feature, the other shrinks the number of features.
 
\subsection{Contributions}
 
This paper develops the quadrature view of kernel feature maps and establishes, in theory
and in experiment, that well-chosen quadrature points yield dramatically better spectral
approximation than random points at equal cost. Concretely:
 
\begin{enumerate}
  \item \textbf{A general spectral error bound for quadrature rules.} For any feature map
  $G$ and any quadrature measure $\nu$ that integrates a subspace $V$ exactly, we bound
  the error in every quadratic form of the Gram matrix, and hence the  spectral
  error  in terms of a quantity determined by the approximation error of $G$. This reduces the problem of spectral kernel
  approximation to a classical question of how well the features can be approximated by the
  integrated subspace, and in particular implies the operator inequality
  $K - \lambda I \ll \widehat{K} \ll K + \lambda I$ once the error is below
  a certain threshold. (See Section~\ref{bhag:mainresults} for details.)
 
  \item \textbf{Instantiation across several kernel families and domains.} We apply the
  general bound to concrete settings:
  Fourier features on $\mathbb{R}^q$, using Gauss-Hermite quadrature, where we adapt
  entire-function approximation estimates to handle the Gaussian weight; and the arc-cosine
  kernels of Cho and Saul~\cite{cho2009kernel} corresponding to the feature maps $G(w,x)=(w\cdot x)_+^\ell$, in both their $\ell = 0$ and $\ell = 1$
  forms. This shows that the framework is not tied to one kernel but covers Euclidean and
  spherical data and several neural-network-inspired kernels.
 
  \item \textbf{Robustness to non-smooth features.} We show, and verify
  empirically, that the spectral advantage of quadrature survives the loss of smoothness.
  For analytic activations (e.g.\ Leaky ReLU) the error decays exponentially in the number
  of points per dimension; for genuinely non-differentiable activations (standard ReLU) the
  decay is slower but still markedly faster than random features, whose error remains large
  and can even \emph{increase} as features are added.
 
  \item \textbf{Storage that is linear, not exponential, in dimension.} For Fourier features
  on a quadrature grid, the product structure of the points lets us store the feature tensor
  as an $M \times q \times n$ object and recover inner products as a product of
  per-coordinate sums. This reduces storage from the naive $O(M n^q)$, exponential in the
  ambient dimension $q$, to $O(M q n)$, making the small-but-expressive feature sets of the
  previous point genuinely practical to compute and store.
 
  \item \textbf{Empirical validation.} On Gaussian-kernel data in $\mathbb{R}^5$ and on
  neural tangent kernels on the sphere $\mathbb{S}^3$ built from Leaky ReLU and ReLU
  activations, we show the spectral convergence of quadrature features matches the theory
  and is exponentially faster than that of random features, often driving the spectral error
  to numerically zero with only a handful of points per dimension while random features
  remain far from a spectral guarantee.
\end{enumerate}
 
Together these results argue that, for the spectral approximation that kernel methods
actually require, the randomness in random features is not a feature but a cost, and one that
structured quadrature can avoid, simultaneously improving accuracy and shrinking the
features that must be stored.

The outline of the remainder of the paper is as follows. In Section~\ref{bhag:mainresults} we state and prove our main result, Theorem~\ref{theo:basictheo}.
The application of this result and numerical validation for Fourier features is discussed in Section~\ref{bhag:fourierfeatures}.
The application to neural tangent kernels based on sigmoidal and ReLU activation functions on the sphere is discussed in Section~\ref{bhag:arccossect} with numerical validation in Section~\ref{sec:numS3}.

\section{Main results}\label{bhag:mainresults}
\subsection{Set up}\label{bhag:setup}
Let $\WW$ be a topological space, $\mu^*$ be a $\sigma$-finite Borel measure supported on $\WW$. 
Let $\XX$ be (possibly different) nonempty set.
A \emph{kernel} on $\XX$ is a function $K:\XX\times\XX\to \RR$ which is  positive semi-definite in the sense that
\be\label{eq:posdef}
\int_\XX\int_\XX K(x,y)d\mu(x)d\mu(y) \ge 0
\ee
for all nonzero signed measures $\mu$ supported on $\XX$.
The \emph{kernel trick} is to find 
 $G:\WW\times\XX\to\CC$ such that $G(\circ, x)$ is $\mu^*$-measurable for all $x,y\in \XX$ and
\be\label{eq:Kkerndef}
K(x,y)=K(\mu^*;x,y)=\int_\WW G(w,x)\ol{G(w,y)}d\mu^*(w), \qquad x,y\in\XX.
\ee
In this paper, we assume that $G$ is known, and define $K$ by \eqref{eq:Kkerndef}.

Our goal is to approximate $K$ by
\be\label{eq:Kapprox}
K(\nu;x,y)=\int_\WW G(w,x)\ol{G(w,y)}d\nu(w), \qquad x,y\in\XX.
\ee
where $\nu$ is a finitely supported measure on $\WW$.
The total variation measure of $\nu$ is denoted by $|\nu|$.

We define
$$
\|f\|_1=\int_\WW |f(w)|d\mu^*(w), \mbox{  if  } f\in L^1(\mu^*), \qquad \|f\|_\infty= \sup_{w\in \WW}|f(w)|, \mbox{  if  } f\in C_0(\WW).
$$
For the  space $\mathfrak{X}=L^1(\mu^*)\bigcap C_0(\WW)$, we define
\be\label{eq:normdef}
\|f\|=\max(\|f\|_1,\|f\|_\infty), \qquad f\in \mathfrak{X}.
\ee
We assume that 
\be\label{eq:Gfamily}
\mathcal{F}=\{G(\circ,x) :x\in\XX\}\subset \mathfrak{X}
\ee
 is compact in $\mathfrak{X}$.

For a linear subspace $V\subset \mathfrak{X}$  we define the degree of approximation by
\be\label{eq:degapprox}
\mathsf{dist}(f, V)=\inf_{P\in V}\|f-P\|, \qquad f\in\mathfrak{X}.
\ee
Obviously, $\mathsf{dist}(f, V)\le \|f\|<\infty$.
Since we assume $\mathcal{F}$ to be compact, 
 there exist  $\epsilon(V)$ and $\tn G\tn$ such that
\be\label{eq:equiconvergence}
\sup_{x\in\XX}\|G(\circ,x)\|\le \tn G\tn,\quad \sup_{x\in\XX}\mathsf{dist}(G(\circ,x), V)\le \epsilon(V).
\ee

A measure $\nu$ supported on $\WW$ is called a (product) \emph{quadrature measure for $V$} if $\nu$ has a finite total variation $|\nu|$, and
\be\label{eq:absquadrature}
\int_\WW |P|^2d\nu=\int_\WW |P|^2d\mu^*, \qquad P\in V,
\ee
equivalently,
\be\label{eq:absquadraturealt}
\int_\WW P\ol{Q}d\nu=\int_\WW P\ol{Q}d\mu^*, \qquad P, Q\in V.
\ee
For example, if $\WW=[-1,1]$, $\mu^*$ is the Lebesgue measure, $n\ge 1$ and $V$ is space of all polynomials of degree $<n$, then the Legendre quadrature formula defines a measure $\nu$ which is a quadrature measure for $V$. 
In general, products of elements of $V$ are not members of $V$, and hence, it is not enough for our purpose to require as usual that $\int Pd\nu=\int Pd\mu^*$ for all $P\in V$.

\subsection{Main theorem}\label{bhag:maintheo}

\begin{theorem}
\label{theo:basictheo}
Let $M, N\ge 1$ be integers, $\nu$ be a quadrature measure for $V$ on $\WW$, with $\mathsf{supp}(\nu)=\{w_1,\cdots, w_N\}$. 
Let $x_1,\cdots, x_M\in\XX$, $\mathbb{K}=[K(x_k,x_j)]_{k,j=1}^M$,  $\widehat{\mathbb{K}}=[K(\nu;x_k,x_j)]_{k,j=1}^M$. Let $\sigma_1\ge \cdots \ge \sigma_M$ and $\tau_1\ge\cdots\ge \tau_M$ be the eigenvalues of $\mathbb{K}$ and $\widehat{\mathbb{K}}$ respectively. 
Let
\be\label{eq:epsilon_def}
\varepsilon=4(1+|\nu|)M\tn G\tn \epsilon(V).
\ee
Then
\be\label{eq:eigenvalueest}
\max_{1\le j\le M}|\sigma_j-\tau_j|\le \varepsilon.
\ee
In particular, the spectral norm of $\mathbb{K}-\widehat{\mathbb{K}}$ is $\le \varepsilon$. 
For $\Lambda \ge \varepsilon$
\be\label{eq:spectralest}
\mathbb{K}-\Lambda\mathbb{I} \ll \widehat{\mathbb{K}}\ll \mathbb{K}+\Lambda\mathbb{I},
\ee
where $\mathbb{I}$ is the identity matrix.
For any $\lambda>0$, $\Delta\ge \varepsilon/(\lambda+\sigma_M)$, 
\be\label{eq:goadestbis}
(1-\Delta)(\mathbb{K}+\lambda \mathbb{I})\ll \widehat{\mathbb{K}}+\lambda \mathbb{I} \ll (1+\Delta)(\mathbb{K}+\lambda \mathbb{I}).
\ee
\end{theorem}


In order to prove the theorem, we need to prove/recall a few lemmas.
The first lemma estimates the error in approximating $\int |f|^2d\mu^*$ by $\int |f|^2d\nu$.
\begin{lemma}\label{lemma:prodquadlemma}
If $\nu$ is a quadrature measure for $V$ and $f, g\in\mathfrak{X}$ then
\be\label{eq:prodquaderr}
\left|\int_\WW f\ol{g}d\mu^*-\int_\WW f\ol{g}d\nu\right|\le 2(1+|\nu|)\left(\|g\|\mathsf{dist}(f, V)+\|f\|\mathsf{dist}(g, V)\right).
\ee 
In particular, setting $f=g$,
\be\label{eq:squarequaderr}
\left|\int_\WW |f|^2d\mu^*-\int_\WW |f|^2d\nu\right|\le 4(1+|\nu|)\|f\|\mathsf{dist}(f, V).
\ee
\end{lemma}

\begin{proof}
We note first that for any $f, g\in \mathfrak{X}$, both $\|fg\|_1\le \|f\|_1\|g\|_\infty$ and $\|fg\|_\infty\le \|f\|_\infty\|g\|_\infty$, and hence, $\|fg\|\le \|f\|\|g\|$.
Let $f, g\in \mathfrak{X}$, $\delta>0$ be arbitrary, and $P, Q\in V$ be such that
$$
\|f-P\|\le \mathsf{dist}(f, V)+\delta, \qquad \|g-Q\|\le \mathsf{dist}(g, V)+\delta.
$$
Then 
$$
\|P\|\le \|f\|+\mathsf{dist}(f,V)+\delta \le 2\|f\|+\delta
$$
so that
\be\label{eq:pf1eqn1}
\begin{aligned}
\|f\ol{g}-P\ol{Q}\|&\le \|(f-P)g\|+\|P(g-Q)\|\le \|f-P\||\|g\|+\|P\|\|g-Q\|\\
&\le \|g\|(\mathsf{dist}(f, V)+\delta)+(2\|f\|+\delta)(\mathsf{dist}(g, V)+\delta).
\end{aligned}
\ee
In view of \eqref{eq:absquadraturealt}, we see that for any $\delta>0$,
$$
\begin{aligned}
\left|\int_\WW f\ol{g}d\mu^* -\int_\WW f\ol{g}d\nu\right| &=\left|\int_\WW (f\ol{g}-P\ol{Q})d\mu^* -\int_\WW (f\ol{g}-P\ol{Q})d\nu\right|\le \|f\ol{g}-P\ol{Q}\|_1+|\nu|\|f\ol{g}-P\ol{Q}\|_\infty \\ &\le 2(1+|\nu|)\bigg(\|g\|(\mathsf{dist}(f, V)+\delta)+\|f\|(\mathsf{dist}(g, V)+\delta)\bigg).
\end{aligned}
$$
\end{proof}

Our next lemma is a simple estimate on the degree of approximation of linear combinations of $f\in\mathfrak{X}$.

\begin{lemma}\label{lemma:linearcomb}
Let $M\ge 1$ be an integer, $f_1,\cdots, f_M \in \mathfrak{X}$, $a_1,\cdots, a_M\in\CC$,
$$
f=\sum_{\ell=1}^M a_\ell f_\ell, \qquad w\in \WW.
$$ 
Then
\be\label{eq:lin_comb_degree_approx}
\mathsf{dist}(f, V)\le \sqrt{M}\left\{\sum_{\ell=1}^M |a_\ell|^2\right\}^{1/2}\max_{1\le \ell \le M}\mathsf{dist}(f_\ell, V).
\ee
\end{lemma}
\begin{proof}
Let $\delta>0$ be arbitrary. For $\ell=1,\cdots,M$, let  $P_\ell\in V$ satisfy $\|f_\ell-P_\ell\|\le \mathsf{dist}(f_\ell, V)+\delta$.
Then
$$
\begin{aligned}
\mathsf{dist}(f, V)&\le \left\|f-\sum_{\ell=1}^M a_\ell P_\ell\right\|\le \sum_{\ell=1}^M |a_\ell|\|f_\ell-P_\ell\|\le \max_{1\le \ell \le M}\left(\mathsf{dist}(f_\ell, V)+\delta\right)\sum_{\ell=1}^M |a_\ell|\\
&\le \sqrt{M}\left\{\sum_{\ell=1}^M |a_\ell|^2\right\}^{1/2}
\max_{1\le \ell \le M}
\left(\mathsf{dist}(f_\ell, V)+\delta\right).
\end{aligned}
$$
\end{proof}

The final ingredient in our proof of Theorem~\ref{theo:basictheo} is the Weyl's inequality (\cite[Exercise before Theorem~6.3.5, p.~367]{horn_johnson_book})

\begin{prop}\label{prop:wielandt}
Let $M\ge 1$ be an integer, $A, B$ be $M\times M$ symmetric matrices, and $\sigma_1\ge \sigma_2\ge\cdots\sigma_M$, $\tau_1\ge \tau_2\ge\cdots\tau_N$ be the eigenvalues  of $A$ and $B$ respectively. 
Let $L$ be the spectral norm of $A-B$. 
Then
\be\label{eq:wieldandt}
\max_{1\le j\le M}|\sigma_j-\tau_j|\le L.
\ee
\end{prop}

\noindent
\textit{Proof of Theorem~\ref{theo:basictheo}.}\\

Let $a_1,\cdots, a_M$ be any complex numbers with $\sum_{\ell=1}^M |a_\ell|^2 =1$, and $f=\sum_{\ell=1}^M a_\ell G(\circ, x_\ell)$.
Then Lemma~\ref{lemma:linearcomb} and \eqref{eq:equiconvergence} yield
$$
\mathsf{dist}(f, V)\le \sqrt{M}\epsilon(V).
$$
Moreover, \eqref{eq:equiconvergence} shows that 
$$\|f\|\le \sum_{\ell=1}^M |a_\ell|\tn G\tn \le \sqrt{M}\tn G\tn.$$
Therefore, \eqref{eq:squarequaderr}   leads to
\be\label{eq:pf2eqn1}
\left|\int_\WW |f|^2d\mu^*-\int_\WW |f|^2d\nu\right|\le 4(1+|\nu|)M\tn G\tn \epsilon(V).
\ee
If $\mu$ denotes either $\mu^*$ or $\nu$, then
$$
\int_W |f|^2d\mu =\sum_{k,j=1}^M a_k\ol{a_j}\int_\WW G(w, x_k)\ol{G(w,x_j)}d\mu(w)=\sum_{k,j=1}^M a_k\ol{a_j} K(\mu;x_k,x_j).
$$
So, \eqref{eq:pf2eqn1} translates into
$$
\sum_{k,j=1}^M a_k\ol{a_j} \left\{K(\mu^*;x_k,x_j)-K(\nu;x_k,x_j)\right\}\le 4(1+|\nu|)M\tn G\tn \epsilon(V).
$$
Thus, the Rayleigh theorem shows that  the spectral norm of $\mathbb{K}-\widehat{\mathbb{K}}$ is bounded above by $4(1+|\nu|)M\tn G\tn \epsilon(V)$.
The estimate \eqref{eq:eigenvalueest} now follows from Proposition~\ref{prop:wielandt}.
The other statements in Theorem~\ref{theo:basictheo} are easy to deduce from \eqref{eq:eigenvalueest}.
\qed

\section{Fourier features}\label{bhag:fourierfeatures}
\subsection{Approximation estimates}\label{bhag:hermiteapprox}
We consider $\WW=\RR^q$, $\mu^*$ to be the Lebesgue measure on $\WW$, $\XX=[-\tau,\tau]^q$ for some $\tau>0$, and
$G(\w,\x)=\exp(i\w\cdot\x)\exp(-|\w|^2/2)$, 
$\w\in\WW$, $\x\in\XX$. 
For $n\ge 1$, we denote by $\mathbb{P}_n^q$ the space of all $q$ variate algebraic polynomials of coordinatewise degree $<n$, and consider
$$
V=\Pi_n^q=\{ \w\mapsto P(\w)\exp(-|\w|^2/2) :  P\in \mathbb{P}_n^q\}.
$$
When $q=1$, we omit the superscript $q$ from these notations.

The orthonormalized Hermite polynomial $h_k$ of degree $k$ is defined recursively by
\bea\label{eq:recurrence}
h_k(x)&:=&\sqrt{\frac{2}{k}}xh_{k-1}(x)-\sqrt{\frac{k-1}{k}}h_{k-2}(x), \qquad k=2,3,\cdots,
\nonumber\\
&&h_0(x):=\pi^{-1/4},\ h_1(t):=\sqrt{2}\pi^{-1/4}x.
\eea
We write 
\be\label{eq:uni_psi_def}
\psi_k(t):=h_k(t)\exp(-t^2/2), \qquad t\in\RR,\ k\in\ZZ_+.
\ee 
The functions $\{\psi_k\}_{k=0}^\infty$ are an orthonormal set on $\RR$ with respect to the Lebesgue measure, and $\Pi_n=\mathsf{span}\{\psi_k : k<n\}$.

 The Hermite polynomial $h_m$ has $m$ real and simple zeros $x_{k,m}$. Writing
\be\label{eq:coatesnum}
 \lambda_{k,m}:=\left(\sum_{j=0}^{m-1} h_j(x_{k,m})^2\right)^{-1},
\ee
it is well known (cf. \cite[Section~3.4]{szego}) that
\be\label{eq:uniquad}
\sum_{k=1}^m \lambda_{k,m}P(x_{k,m})= \int_\RR P(w)\exp(-w^2)dw, \qquad P\in \mathbb{P}_{2m}.
\ee
Thus, the measure $\nu_m$ that associates the mass $\lambda_{k,m}\exp(x_{k,m}^2)$ with each $x_{k,m}$ is a quadrature measure for $\Pi_m$.

It is also known (cf. \cite[Theorem~8.2.7]{mhasbk}, applied with $p=2$, $b=0$) that
\be\label{eq:wtvarbd}
|\nu_m|=\sum_{k=1}^m \lambda_{k,m}\exp(x_{k,m}^2) \ls m^{1/2}.
\ee
For $f\in L^1(\RR)$, $n\ge 1$, we define
\be\label{eq:hermitexp}
\begin{aligned}
\hat{f}(k)&=\int f(y)\psi_k(y)dy, \qquad k=0,1, \cdots,\\
s_n(f)(x)&=\sum_{k=0}^{n-1} \hat{f}(k)\psi_k(x).
\end{aligned} 
\ee
During the proof of \cite[Theorem~3.2]{bitrep} (see the first displayed equation on p.~474), we have proved that if $f$ is an entire function of finite exponential type $\le \tau$, then with $g(w)=f(w)\exp(-w^2/2)$, we have for every $p\in [1,\infty]$,
$$
\|g-s_n(g)\|_p \ls \tau^n/\sqrt{n!}.
$$
(In \cite{bitrep}, we need to take $\alpha=2$, $\lambda=1$, so that $\rho$ as defined in \cite[Eqn. (2.5)]{bitrep} is $1$. The term $M_n/2^n\ls 1$, since $M_n^{1/n}\to 1$ as $n\to\infty$.)
Applying this result with $g(w)=g_x(w)=\exp(iwx)\exp(-w^2/2)$, $|x|\le\tau$, we see that
$$
\sup_{g_x}\inf_{P\in\Pi_n}\{\|g-P\|_1+\|g-P\|_\infty\}\ls \tau^n/\sqrt{n!}.
$$
Using a simple tensor product argument as in \eqref{eq:pf1eqn1}, we conclude that
\be\label{eq:epsilonbdwt}
\epsilon(\Pi_n^q)\ls \tau^n/\sqrt{n!}.
\ee

\subsection{Efficient Storage and Computation of Quadrature Fourier Kernel}
In the instance that $\WW=\XX=\RR^q$, $G(\w,\x)=\exp{(i \w\cdot\x-\|\w\|^2/2)}$, the storage complexity of the $n^q$ quadrature Fourier features can be vastly reduced by using tensors.   Let $\x_{j,k}$ be the $k^{th}$ feature of the $j^{th}$ point, and let $\w_{\ell,k}$ be the $\ell^{th}$ quadrature point in the $k^{th}$ feature.  The benefit of this framework is that $\WW$ lies on a grid, so a particular quadrature point can be decomposed into the Cartesian product of the quadrature points in each feature.  Given this, we can define
\begin{align*}
\Phi_{j,k,\ell} = \frac{\sqrt{a_\ell}}{\pi^{1/4}}\textnormal{exp}\left( i x_{j,k}  w_{\ell,k} -w_{\ell,k}^2/2\right), \qquad j\in\{1,...,M\}, \, k\in\{1, ..., q\}, \, \ell\in\{1, ...,n\},
\end{align*}
where $a_{\ell}$, for $\ell=1,\ldots,n$, are Gauss-Hermite quadrature weights.
A simple observation of the exponential yields that $\Phi$ can be stored simply as a $M\times q \times n$ tensor $\Phi_{j,k,\ell} = \frac{a_\ell}{\pi^{1/2}}\textnormal{exp}\left( i x_{j,k}  w_{\ell,k} -w_{\ell,k}^2/2\right)$, since
\begin{align*}
\langle \Phi_{j,\cdot,\cdot}, \Phi_{j',\cdot,\cdot} \rangle  = \prod_k \sum_\ell \frac{a_\ell}{\pi^{1/2}}\textnormal{exp}\left( i x_{j,k}  w_{\ell,k} -w_{\ell,k}^2/2\right) \textnormal{exp}\left( - i x_{j',k}  w_{\ell,k} -w_{\ell,k}^2/2\right)
\end{align*}
The benefit of this is that the storage complexity of ${\mathcal O}(Mqn)$ is significantly better than the naive RFF storage of $\mathcal{O}(Mn^q)$ that does not exploit the grid structure of the weights.  

\subsection{Experimental Validation}

We examine the spectral approximation error as a function of the number of features kept, comparing between the quadrature kernel and a random Fourier feature kernel.  For each, generate $n^q$ features where $n\in [2,15]$.  For the quadrature points, this is $n$ quadrature points per dimension.  

\begin{table}[]
    \centering
    \begin{tabular}{|c|c|c|c|}
    \hline
$n$ & $n^q$ &  $F_M$ & $\tilde{F}_M$   \\ 
\hline
%
4 & 1024 & 5.656525e-02 & 6.583937e+00 \\
5 & 3125 & 2.326465e-03 & 3.529518e+00 \\
6 & 7776 & 8.267983e-05 & 2.180179e+00 \\
7 & 16807 & 2.578723e-06 & 1.520654e+00 \\
8 & 32768 & 7.155622e-08 & 1.003574e+00 \\
9 & 59049 & 1.786809e-09 & 7.998515e-01 \\
10 & 100000 & 4.054818e-11 & 6.485978e-01 \\
11 & 161051 & 1.914069e-12 & 4.963733e-01 \\
12 & 248832 & 2.291393e-12 & 4.029682e-01 \\
13 & 371293 & 2.768165e-13 & 3.298760e-01 \\
\hline
\end{tabular}
\caption{The eigenvalue discrepancies with $M=1000$ and  various numbers of quadrature points. The random experiments were repeated $50$ times and $\tilde{F}_M$ is the averaged value.}
    \label{tab:Gaussian M1000 RFF quad}
\end{table}

\begin{table}[]
    \centering
    \begin{tabular}{|c|c|c|}
    \hline
    $M$ & $F_M$ & $\tilde{F}_M$ \\
    \hline
 125    &  1.409252e-13 & 6.267601e-02\\
 250    &  3.615065e-13 & 1.209672e-01\\
 500    &  9.514168e-13 & 2.323356e-01\\
 1000   &  1.132142e-12 & 4.864472e-01\\
  \hline
\end{tabular}
\caption{Eigenvalue discrepancies with $n=11$ (hence $n^q = 161051$) with different values of $M$. The value $\tilde{F}_M$ is averaged over $50$ trials.}
\label{tab:Gaussian different M}
\end{table}
We define the matrices $\mathbb{K} = [K(\bsx_k,\bsx_j)]_{k,j=1}^M$,
with $K$ being the Gaussian kernel with bandwidth $\sigma$, that
is
\[
K(\bsx, \bsy) = \exp \left(-\frac{\|\bsx-\bsy\|^2}{2\sigma^2}\right).
\]
In the numerical experiments, we set the bandwidth $\sigma=1$, and take $q=5$.
Computationally, we use tensor Gauss-Hermite quadrature on $\mathbb{R}^q$, see
\cite{gautschi2004orthogonal,olver2010quadrature}, where each quadrature points on each dimension is scaled by a factor of $\sqrt{2}/\sigma$, and the weights adjusted accordingly. 
Equivalently with the theory presented in Section~\ref{bhag:hermiteapprox}, we use the kernel $G(\w,\x)=\exp(-i\w\cdot\x)$ and absorb the Hermite weight in $\mu^*$, so that the integral is with respect to $\exp(-\|\w\|^2)$
rather than with respect to the Lebesgue measure. With these choices, 
\[
\begin{aligned}
K(\bsx,\bsy)&=\exp\left(-\frac{\|\bsx - \bsy\|^2}{2\sigma^2}\right)=\frac{1}{\pi^{q/2}}\int_{\mathbb{R}^q} \exp\left( i \frac{\sqrt{2}}{\sigma} \bsw \cdot (\bsx - \bsy) \right) e^{-|\bsw|^2} \, d\bsw\\
&=  \frac{1}{\pi^{q/2}}\int_{\mathbb{R}^q} G(\w,\x)\ol{G(\w,\y)}e^{-|\bsw|^2} \, d\bsw, \qquad \bsx,\bsy\in\RR^q.
\end{aligned}
\]
We also define $\widehat{\mathbb{K}} = [K(\nu;\bsx_k,\bsx_j)]_{k,j=1}^M$,
with
\[
K(\nu;\bsx,\bsy) = \frac{1}{\pi^{q/2}}\int_{\R^q} G(\bsw,\bsx)\overline{G(\bsw,\bsy)} d\nu(\bsw),
 \qquad \bsx, \bsy \in \R^q,
\]
Finally, the matrix $\widetilde{\mathbb{K}} = [K(\mu;\bsx_k,\bsx_j)]_{k,j=1}^M$ is defined with 
\[
K(\mu;\bsx,\bsy) = \int_{\R^q} G(\bsw,\bsx)\overline{G(\bsw,\bsy)} d\mu(\bsw),
\quad \bsx, \bsy \in \R^q.
\]
where $\mu$ is a discrete measure supported on a random set of $N=n^q$ points. Let $\lambda_1,\ldots,\lambda_M$ be the eigenvalues of the matrix $\mathbb{K}$, $\widehat{\lambda}_1,\ldots,\widehat{\lambda}_M$ be the eigenvalues of the matrix $\widehat{\mathbb{K}}$, and let
$\widetilde{\lambda}_1,\ldots,\widetilde{\lambda}_M$ be the eigenvalues of the matrix $\widetilde{\mathbb{K}}$, all arranged in a decreasing order.
Let the eigenvalue discrepancies be defined by
\[
F_M = \left( \sum_{k=1}^M |\lambda_k-\widehat{\lambda}_k|^2\right)^{1/2},\quad
\widetilde{F}_M = \left( \sum_{k=1}^M |\lambda_k-\widetilde{\lambda}_k|^2\right)^{1/2}.
\]
We note that, the Hoffman--Wielandt inequality for symmetric matrices \cite{HoffmanWielandt1953} gives
\[
F_{M} \le \| \mathbb{K} - \widehat{\mathbb{K}} \|_{\rm F},\quad
\widetilde{F}_{M} \le \| \mathbb{K} - \widetilde{\mathbb{K}} \|_{\rm F},
\]
where $\|\cdot\|_{\rm F}$ denotes the Frobenius norm of a matrix.
Table~\ref{tab:Gaussian M1000 RFF quad} examines the eigenvalue discrepancies with a fixed number of data points $M$ and various numbers of quadrature points while in Table~\ref{tab:Gaussian different M} the number of data points varies when the number of quadrature points is fixed.

Another method of measuring error is to examine the eigenvalues of the matrix 
\begin{align*}
\widehat{\mathbb{K}} + \lambda I - (1-\Delta) (\mathbb{K} + \lambda \mathbb{I})
\end{align*}
Let $\lambda_j$ be the $j^{th}$ eigenvalue, we compute $S_{\textnormal{lower}} = \sum_j |\lambda_j| \mathsf{1}_{\lambda_j < 0}$.  We equivalently compute $S_{\textnormal{upper}}$ from the matrix $(1+\Delta) (\mathbb{K} + \lambda I) - (\widehat{\mathbb{K}} + \lambda I)$, and a final score $S_M = S_{\textnormal{lower}} +S_{\textnormal{upper}}$.  This gives a score of how far away the associated differences are from being positive semi-definite.  

Figure \ref{fig:gaussianA} examines the plot of the number of quadrature points $N$ against $S_M$ for $M=1000$ points sampled from Gaussian in $q=5$.  Here the true kernel is a Gaussian with bandwidth $\sigma=1$, and both the quadrature and random features $w_j$ are used to create a feature space $z(x) = [e^{i \langle w_j,x\rangle}]_{j=1}^{n^q}$.
 As is clear, the quadrature spectral error decays exponentially fast for almost any $(\lambda, \Delta)$ and often is exactly 0 by $n=10$.  This is compared to random Fourier features, whose error is still quite large by $n=15$ regardless of $(\lambda, \Delta)$.

\begin{figure}
\centering
\includegraphics[width=0.8\textwidth]{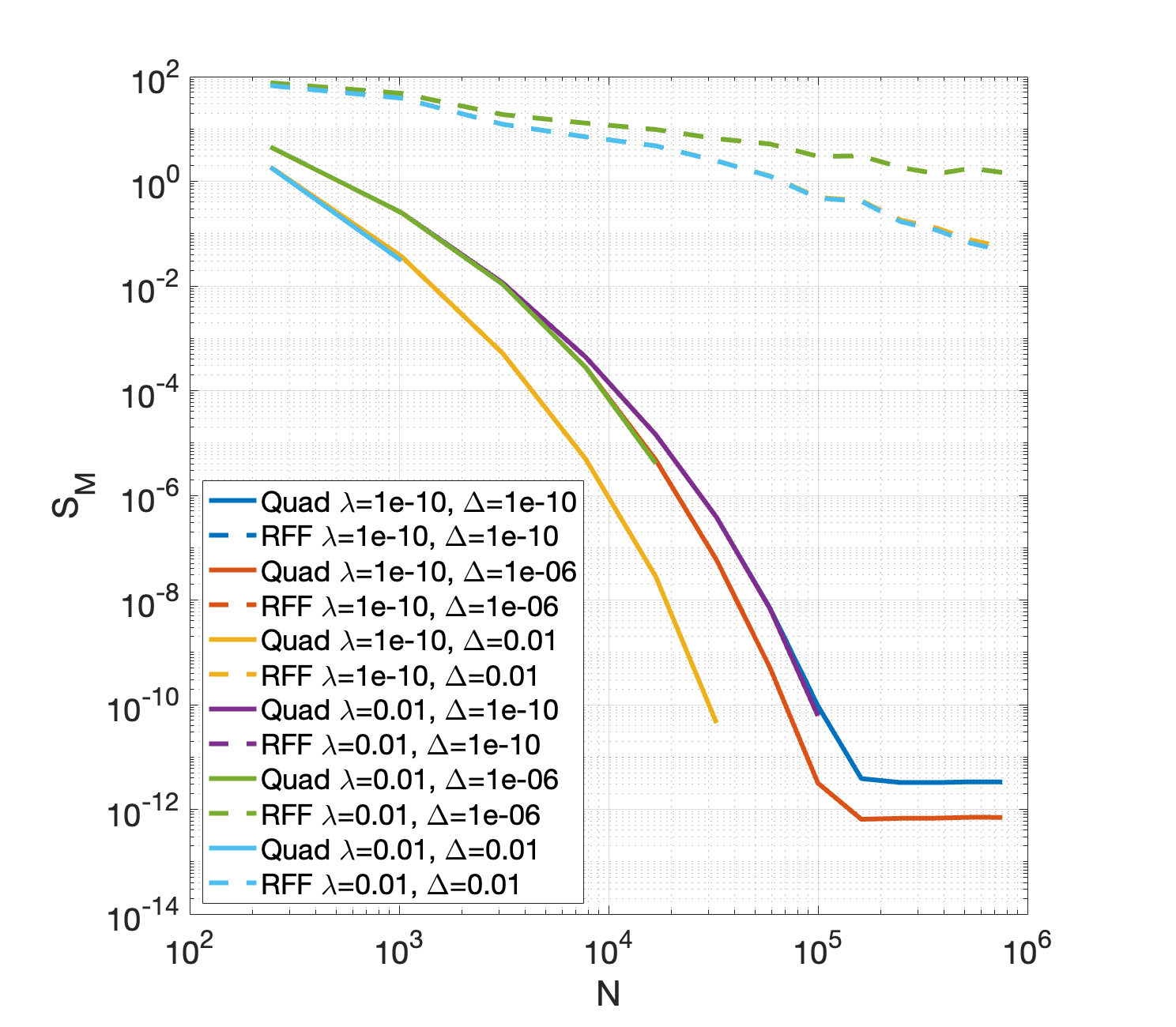}
\caption{Spectral scores for kernel built on $M=1000$ points sampled from Gaussian in $q=5$ dimension.  Compares quadrature features vs random Fourier features.  Activation is complex exponential, comparison to true kernel. Some curves stop short due to their remaining values being identically zero.}\label{fig:gaussianA}
\end{figure}

\section{The arc-cosine kernel on the unit sphere}\label{bhag:arccossect}
\subsection{Definitions}\label{bhag:arccoskerndef}
The arc-cosine kernel, see \cite{Cho2009KernelMF}, is described below.
Let $\bsx, \bsy \in \R^{q+1}$ and
let $(z)_{+} = z$ if $z \ge 0$ and $0$ otherwise. 
In the special case, $(z)^0_{+} = 1$ if $z >0$ and $0$ otherwise. The $\ell$th order arc-cosine kernel function is defined by
\[
k_\ell(\bsx,\bsy) = 
\frac{2}{(2\pi)^{(q+1)/2}} 
\int_{\R^{q+1}}
(\bsw \cdot \bsx)_{+}^\ell (\bsw \cdot \bsy)_{+}^\ell  e^{-\|\bsw|^2/2}d \bsw
\]
The kernel $k_\ell$ has an explicit form, which is given by
\[
k_\ell(\bsx,\bsy) = \frac{1}{\pi} \|\bsx\|^\ell
\|\bsy\|^\ell J_\ell(\theta),
\text{ with }
\cos \theta = \frac{\bsx \cdot \bsy}{ \|\bsx\| \|\bsy\|},
\]
and the function $J_\ell(\theta)$ is given by
\[
J_\ell(\theta) = (-1)^\ell (\sin\theta)^{2\ell+1}
\left(\frac{1}{\sin\theta} \frac{\partial}{\partial\theta} \right)^\ell
\left(\frac{\pi-\theta}{\sin\theta} \right).
\]
A few interesting cases are
\begin{align}
   J_0(\theta) &= \pi - \theta\\
   J_1(\theta) &= \sin\theta + (\pi-\theta)\cos\theta.
\end{align}
If we restrict $\bsx, \bsy$ to the unit sphere $\bS^q$ and let $\bsw = r \hat{\bsw}$, using the identity
\[
\int_0^\infty r^{2\ell+q} e^{-r^2/2}dr
= 2^{\ell+(q-1)/2} \Gamma\left(\ell + \frac{(q+1)}{2}\right),
\]
we obtain
\begin{align}
K_\ell(\bsx,\bsy) &= \frac{2}{(2\pi)^{(q+1)/2}}\int_0^\infty
\int_{\bS^q} (\hat{\bsw} \cdot \bsx)_{+}^\ell(\hat{\bsw} \cdot \bsy)_{+}^\ell
r^{2\ell+q} e^{-r^2/2} d\mu_q^*(\hat{\bsw}) dr \nonumber\\
&=  \frac{2}{(2\pi)^{(q+1)/2}} 
2^{\ell+(q-1)/2} \Gamma\left(\ell + \frac{(q+1)}{2}\right)
\int_{\bS^q} (\hat{\bsw} \cdot \bsx)_{+}^\ell (\hat{\bsw} \cdot \bsy)_{+}^\ell d\mu_q^*(\hat{\bsw})\nonumber\\
&=  C_{q,\ell} \int_{\bS^q} (\hat{\bsw} \cdot \bsx)_{+}^\ell (\hat{\bsw} \cdot \bsy)_{+}^\ell d\mu_q^*(\hat{\bsw}),\label{eqn:K-ell}
\end{align}
where $\mu_q^*$ denotes the volume measure of $\SS^q$, and
\[
C_{q,\ell} = \frac{2^\ell}{\pi^{(q+1)/2}} 
\Gamma \left(\ell + \frac{(q+1)}{2}\right).
\]
Suppose $\nu$ is a discrete measure on $\bS^q$, we approximate $K_\ell$ by
\begin{equation}
  K_\ell(\nu; \bsx, \bsy)
  = C_{q,\ell} \int_{\bS^q} 
  (\hat{\bsw} \cdot \bsx)_{+}^\ell (\hat{\bsw} \cdot \bsy)_{+}^\ell d\nu(\hat{\bsw}),
  \quad \bsx, \bsy \in \bS^q,
\end{equation}
where $\nu$ is a finitely supported measure
on $\bS^q$ with cardinality $N=\#\text{supp}(\nu)$.
\subsection{Quadrature formulas}\label{bhag:sphquadrature}

A point $\bsx = (x_1, x_2, \ldots, x_{q+1}) \in \bS^q$ is parameterised by $q$ angles \newline
$(\theta_1, \theta_2, \ldots, \theta_q)$ with domain $\theta_1 \in [0, 2\pi)$ and $\theta_m \in [0, \pi]$ for $m = 2, \ldots, q$:
\begin{equation}\label{eq:sphcoords}
\begin{aligned}
 x_1 &= \sin(\theta_1) \prod_{m=2}^q \sin(\theta_m), \\
 x_2 &= \cos(\theta_1) \prod_{m=2}^q \sin(\theta_m), \\
 x_m &= \cos(\theta_{m-1}) \prod_{l=m}^q \sin(\theta_l), \quad \text{for } 3 \le m \le q+1,
\end{aligned}
\end{equation}
where by convention $\prod_{l=q+1}^q \sin(\theta_l) \equiv 1$.
More succinctly, any point on $\SS^q$ ($q\ge 2$) can be written in the form $\x=(\sin\theta \omega, \cos\theta)$, where $\omega\in \SS^{q-1}$, $\theta\in [0,\pi]$. 
In \cite{stroud1971approximate}, it is proposed to construct a quadrature formula on the sphere $\SS^q$ using the quadrature formulas for different ultraspherical (Gegenbauer) weights for different dimensions. 
We propose a similar construction, except that we use only two quadrature formulas after the initial FFT along $\theta_1$, denoted by $\nu_{n,1}$.

It is convenient to express our quadrature formula using a measure notation. 
Let $\nu_{n,q}$ be a discrete measure on $\SS^q$ that integrates spherical polynomials of degree $<2n$ exactly, $\nu_n^L$ (resp. $\nu_n^T$) be the quadrature measure that integrates algebraic polynomials of degree $<2n$ on $[-1,1]$ with respect to the Lebesgue measure (resp. Chebyshev weight).
Thus, $\nu_n^L$ is supported on the zeros of degree $n$ Legendre polynomials and $\nu_n^T$ is supported on the zeros of degree $n$ Chebyshev polynomials of first kind.

A spherical polynomial of degree $<2n$ has the form
$$
P(\x)=\sum_{0\le \ell <2n}\sin^\ell\theta Q_{2n-\ell}(\cos\theta)Y_\ell(\omega),
$$
where $Q_{2n-\ell}$ is a univariate algebraic polynomial of degree $<2n-\ell$ and $Y_\ell $ is a homogeneous harmonic spherical polynomial of degree $\ell$.
In particular, if $\mu_q^*$ is the volume measure on $\SS^q$, then
\be\label{eq:exactint}
\begin{aligned}
\int_{\SS^q} P(\x)d\mu_q^*(\x)&=\int_{0}^\pi Q_{2n-\ell}(\cos\theta)\sin^{\ell+q-1}(\theta)\int_{\SS^{q-1}}Y_\ell(\omega)d\mu_{q-1}^*(\omega)\\
&=\int_{0}^\pi Q_{2n}(\cos\theta)\sin^{q-1}(\theta)d\theta=\int_{-1}^1 Q_{2n}(x)(1-x^2)^{q/2-1}dx.
\end{aligned}
\ee
If $q$ is even ($\ge 2$), then $(1-x^2)^{q/2-1}$ is a polynomial of degree $q-2$, and hence,
$$
\int_{-1}^1 Q_{2n}(x)(1-x^2)^{q/2-1}dx=\int_{-1}^1 Q_{2n}(x)(1-x^2)^{q/2-1}d\nu_{n+q/2-1}^L(x).
$$
Otherwise, 
$$
(1-x^2)^{q/2-1}=(1-x^2)^{q/2-1/2}(1-x^2)^{-1/2},
$$
where $(1-x^2)^{q/2-1/2}$ is a polynomial of degree $q-1$, and
$$
\int_{-1}^1 Q_{2n}(x)(1-x^2)^{q/2-1}dx=\int_{-1}^1 Q_{2n}(x)(1-x^2)^{q/2-1/2}d\nu_{n+(q-1)/2}^T(x).
$$
Since
$$
\int_{\SS^{q-1}}Y_\ell(\omega)d\mu_{q-1}^*(\omega)=\int_{\SS^{q-1}} Y_\ell(\omega)d\nu_{n,q-1}(\omega), \qquad \ell=0,\cdots, n-1,
$$
we arrive at our quadrature measure $\nu_{n,q}$:
$$
\nu_{n,q}=\begin{cases}
\nu_{n, q-1}\nu_{n+q/2-1}^L, &\mbox{ if $q$ is even},\\
\nu_{n, q-1}\nu_{n+q/2-1/2}^T, &\mbox{ if $q$ is odd}.
\end{cases}
$$

For the convenience of the reader, we elaborate further in the case of $\SS^3$.
A parametrisation of a point $\bsx \in \bS^3$ is given by
\begin{align*}
 x_1 &= \sin(\theta_1) \sin(\theta_2) \sin(\theta_3),\\
 x_2 &= \cos(\theta_1) \sin(\theta_2) \sin(\theta_3),\\
 x_3 &= \cos(\theta_2) \sin(\theta_3),\\ 
 x_4 &= \cos(\theta_3), \quad \theta_1\in [0,2\pi), \; \theta_2 \in [0,\pi], \; \theta_3 \in [0,\pi]
\end{align*} 
The surface measure on $\bS^3$ is
\[
   dS = \sin^2(\theta_3)\sin(\theta_2) d\theta_3 d\theta_2 d\theta_1.
\]
Let $y = \cos(\theta_3)$, so $y \in  [-1,1]$ and 
\[     
dy = -\sin(\theta_3) d\theta_3  = -\sqrt{1-y^2} d\theta_3
\]
Thus
\begin{align*}     
   dS &= (1-y^2)\sin(\theta_2) \frac{(-dy)}{\sqrt{1-y^2}} d\theta_2 d\theta_1
\end{align*}
Now let $z = \cos(\theta_2)$, then
$dz = -\sin(\theta_2) d\theta_2$.
So the surface measure of $\bS^3$ in $\R^4$
\[
    dS = (1-y^2) \frac{dy}{\sqrt{1-y^2}} dz d\theta_1
\]
Choose $y_k = \cos(\theta_{3,k})$ to be the Chebyshev points of degree $n+2$, that is
\[
    \theta_{3,k}  = \frac{(2k-1)\pi}  {2(n+2)}: k=1,...,n+2.
 \]   
Let $\{z_j=\cos(\theta_{2,j}), w_j \}$ for $j=1,\ldots,n$ be the zeros of Gauss-Legendre polynomials of degree $n$ and the corresponding Gaussian weights and let $\theta_{1,i} = 2\pi i /(2n)$ for $i=0,\ldots,2n-1$ which are equally spaced on $[0, 2\pi]$.

The quadrature $Q_n$ for a continuous function $f \in C(\bS^3)$ is defined by
\be\label{eq:explicit_s3_quadrature}
 Q_n(f) = \sum_{i=0}^{2n-1} \sum_{j=1}^n \sum_{k=1}^{n+2} w_{i,j,k} f(\bsy_{i,j,k}),
\ee
where, with
\begin{align*}
\bsy_{i,j,k} =&(\sin \theta_{1,i} \sin \theta_{2,j} \sin \theta_{3,k}, \\
               &      \cos \theta_{1,i} \sin \theta_{2,j} \sin \theta_{3,k},\\
               &      \cos \theta_{2,j} \sin\theta_{3,k},\\
               &      \cos\theta_{3,k}), 
            \quad\qquad i,j,k = 1,\ldots,n,
\end{align*}
we have
\be\label{eq:def:Qn_on_Sq}
w_{i,j,k} = w_{j} \frac{\pi}{(n+2)}  \frac{\pi}{n} \sin^2\theta_{3,k}.
\ee
We have
\begin{equation}\label{eq:sumw}
	\sum_{i,j,k} w_{i,j,k} = 2 \pi^2 = |\bS^3|.	
\end{equation}
Let $\nu$ be the finite measure on $\bS^3$ defined by the quadrature $Q_n$, that is
\[
\int_{\bS^3} f(\bsx) d\nu(\bsx) = Q_n (f), \quad f \in C(\bS^3).
\]
\section{Numerical experiments on $\bS^3$}\label{sec:numS3}
Let $X  = \{\bsx_1, \ldots, \bsx_M\}$ be a set of scattered points on the sphere $\bS^q$ and
let
\begin{equation}\label{def:Kex Knu}
\mathbb{K}_{\ell}=[K_{\ell}(\bsx_k,\bsx_j)]_{k,j=1}^M,  \quad \widehat{\mathbb{K}}_{\ell}=[K_{\ell}(\nu;\bsx_k,\bsx_j)]_{k,j=1}^M,
\quad \ell=0,1.
\end{equation}
Let $\lambda_{1,\ell},\cdots,\lambda_{M,\ell}$ be the eigenvalues of $\mathbb{K}_{\ell}$, and $\hat{\lambda}_{1,\ell},\cdots,\hat{\lambda}_{M,\ell}$ be the eigenvalue of $\widehat{\mathbb{K}}_{\ell}$, both arranged in a decreasing order. We compute the eigenvalue discrepancies by
\begin{equation}\label{eqn:SM}
F_{M,\ell} = \left(\sum_{k=1}^M (\lambda_{k,\ell} - \hat{\lambda}_{k,\ell})^2\right)^{1/2}, \quad \ell=0,1.    
\end{equation}

Let $\mu$ be a finite measure with equal weights and 
$\#\text{supp}(\mu)=\#\text{supp}(\nu)=N$ but the quadrature points are chosen randomly. We also compute the matrices 
\begin{equation}\label{def:Kran}
\widetilde{\mathbb K}_\ell = [K_\ell(\mu;\bsx_k,\bsx_j)]_{k,j=1}^M, \quad\ell=0,1,
\end{equation}
and the corresponding eigenvalue discepancies $\tilde{F}_{M,\ell} = 
\left(\sum_{k=1}^M (\lambda_{k,\ell}-\widetilde{\lambda}_{k,\ell})^2 \right)^{1/2}$ for $\ell=0,1$.

By exploiting the grid structure of the quadrature \eqref{eq:explicit_s3_quadrature} and \eqref{eq:def:Qn_on_Sq}, the storage of the quadrature is of order ${\mathcal O}(qn)$ instead of ${\mathcal O}(n^q)$ using random points. 

With $q=3$, and $\ell=0, 1$, the formula  \eqref{eqn:K-ell} is reduced to
\[
K_0(\bsx,\bsy) = \frac{1}{\pi^2}
\int_{\bS^3} (\hat{\bsw} \cdot \bsx)_{+}^0 (\hat{\bsw} \cdot \bsy)_{+}^0 d\mu_3^*(\hat{\bsw}),
\quad \bsx, \bsy \in \bS^3.
\]
and 
\[
K_1(\bsx,\bsy) = \frac{4}{\pi^2}
\int_{\bS^3} (\hat{\bsw} \cdot \bsx)_{+} (\hat{\bsw} \cdot \bsy)_{+} d\mu_3^*(\hat{\bsw}),
\quad \bsx, \bsy \in \bS^3.
\]
Let $g(x) = x/(1 + e^{-x})$, we also define the following kernel
\begin{equation}\label{defLKg}
    K_g(\bsx,\bsy) = \int_{\bS^3} g( \hat{\bsw} \cdot \bsx )
    g( \hat{\bsw} \cdot \bsy ) d\mu_3^*(\hat{\bsw}), \quad \bsx, \bsy \in \bS^3.
\end{equation}
We note that for the family $\{(\bsw\cdot\bsx)\}$, it is shown in \cite{sphrelu} that $\epsilon(\Pi_n^q)=\mathcal{O}(n^{-1})$. 
Since $g$ is analytic on $[-1,1]$ with the nearest singularities at $\pm i\pi$, the well known Bernstein theorem \cite[Chapter~7, Theorem~8.1]{devlorbk} can be used to show that for the family $\{g(\bsw\cdot\bsx)\}$, $\epsilon(\Pi_n^q)=\mathcal{O}(\rho^n)$ for some $\rho$ with $0<\rho<1$.
\subsection{Numerical validation}\label{bhag;numvalidation}
We now generate a set of $M$ scattered points on $\bS^3$ using an algorithm which partitions a finite-dimensional unit sphere into regions of equal area~\cite{leopardi2006partition}. We then compute the matrices $\mathbb{K}_0$ and $\widehat{\mathbb{K}}_0$ as defined in \eqref{def:Kex Knu} and the eigenvalue discrepancy $F_{M,0}$ as defined in \eqref{eqn:SM}. 

Let $\mu$ be a finite measure with equal weights, and let 
$N=\#\text{supp}(\mu)=\#\text{supp}(\nu)$, with the quadrature points chosen randomly. 
We then compute the matrix $\widetilde{\mathbb{K}}_0$ as in \eqref{def:Kran} and the eigenvalue discrepancy $\tilde{F}_{M,0}$.
\begin{table}[h]
\centering
\begin{tabular}{|c|c|c|c|c|c|c|}
\hline
$M$  &  $\|\mathbb{K}_0-\widehat{\mathbb{K}}_0\|$ & $F_{M,0}$& $\tilde{F}_{M,0}$ &$\|\mathbb{K}_1-\widehat{\mathbb{K}}_1\|$ & $F_{M,1}$& $\tilde{F}_{M,1}$  \\
\hline
100  & 0.0137 & 0.0120 & 0.0608 & 0.0039  & 8.0235e-04 & 0.0819  \\ 
200 & 0.0643 & 0.0374 & 0.1140 & 0.0053 & 8.2850e-04 & 0.1643\\ 
400 & 0.1421 & 0.0480 & 0.2191 & 0.0099 & 0.0012& 0.3302 \\
800 & 0.1086 & 0.0600 & 0.4513 & 0.0111 & 0.0024 & 0.6573\\
1600 & 0.2366& 0.0889 & 0.8698 & 0.0155 & 0.0016 & 1.3117 \\
\hline
\end{tabular}
\caption{Spectral errors and eigenvalue discrepancies between $\mathbb{K}_0$
and $\widehat{\mathbb K}_0$, between $\mathbb{K}_1$
and $\widehat{\mathbb K}_1$ 
with $N=\#\text{supp}(\nu)=134400$ quadrature points vs.
$N=\#\text{supp}(\mu)=134400$ random points on $\bS^3$. The random experiments were repeated $100$ times, then $\tilde{F}_{M,0}$ and $\tilde{F}_{M,1}$ are the average values.}
\label{tab:sp err K0 and K1 combined}
\end{table}


We then repeat the same experiments with
$\mathbb{K}_1=[K_1(\bsx_k,\bsx_j)]_{k,j=1}^M$,  $\widehat{\mathbb{K}}_1=[K_1(\nu;\bsx_k,\bsx_j)]_{k,j=1}^M$
and list the numerical results in Table~\ref{tab:sp err K0 and K1 combined}.

We now fix $M=500$ and vary the quadratures. We carried out experiments for both kernels $K_0$ and $K_1$, 
see Table~\ref{tab:sp err K0 fixed M and K1 fixed M} for the numerical results.
\begin{table}[h]
\centering
\begin{tabular}{|l|l|c|c|c|c|}\hline
$n$ & $N=|\nu|$ &  $F_{M,0}$& $\tilde{F}_{M,0}$
&$F_{M,1}$& $\tilde{F}_{M,1}$\\
\hline
10 & 2400 & 0.6076  &  2.5174  &  0.0458 &   2.0158 \\
20 & 17600 & 0.1315  &  0.7351 &   0.0068 &   0.6416 \\
30 & 57600 & 0.0412  &  0.4273 &   0.0040 &   0.3877 \\
40 & 134400& 0.0346  &  0.2814 &   0.0009 &   0.2456 \\
50 & 260000& 0.0190  &  0.2052 &   0.0008 &   0.1720 \\ 
\hline
\end{tabular}
\caption{Eigenvalue discrepancies between $\mathbb{K}_0$
and $\widehat{\mathbb K}_0$
with fixed $M=500$ and various quadrature points vs. random points on $\bS^3$. The random experiments were repeated $100$ times and 
$\tilde{F}_{M,0}$, $\tilde{F}_{M,1}$ are the average values.}
    \label{tab:sp err K0 fixed M and K1 fixed M}
\end{table}

We also carried out similar experiments on large values of $M$ and report the results in Table~\ref{tab:large M for K0 and K1}.
\begin{table}[h]
\centering
\begin{tabular}{|l|l|l|c|c|c|c|}\hline
$M$ & $n$ & $N=|\nu|$ &  $F_{M,0}$& $\tilde{F}_{M,0}$
&$F_{M,1}$& $\tilde{F}_{M,1}$\\
\hline
20000 & 10 & 2400 & 20.2898 &  62.4916 &   0.1778 & 204.5336\\
&20 & 17600 & 4.0535 &  66.2940 &   0.0177 & 100.2297\\
\hline
40000 & 10 & 2400 & 40.4005 & 233.2059&    0.3487 & 332.4455\\
      & 20 &17600 & 8.1239 & 120.3121&    0.0275&  170.0118 \\
\hline      
\end{tabular}
\caption{Eigenvalue discrepancies between $\mathbb{K}_0$
and $\widehat{\mathbb K}_0$, $\mathbb{K}_1$
and $\widehat{\mathbb K}_1$ 
with different values of $M$ and different quadrature points vs. random points on $\bS^3$.}
    \label{tab:large M for K0 and K1}
\end{table}


\newpage
We now examine the eigenvalues of the matrix 
\begin{equation}\label{equ:matU}
\widehat{\mathbb{K}}_0 + \lambda \mathbb{I} - (1-\Delta) (\Kex_0 + \lambda \mathbb{I}),    
\end{equation}
where $\widehat{\mathbb{K}}_0$ and $\Kex_0$ are defined in equation \eqref{def:Kex Knu}.
Let $\lambda_{i,0}$ be the $i^{th}$ eigenvalue, we compute $S_{\textnormal{lower}} = \sum_i |\lambda_{i,0}| \mathsf{1}_{\lambda_{i,0} < 0}$.  We equivalently compute $S_{\textnormal{upper}}$ from the matrix 
\begin{equation}\label{equ:matL}
(1+\Delta) (\Kex_0 + \lambda \mathbb{I}) - (\widehat{\mathbb{K}}_0 + \lambda \mathbb{I}),    
\end{equation}
and a final score $S_M = S_{\textnormal{lower}} +S_{\textnormal{upper}}$.  This score measures how far the associated differences are from being positive semi-definite. Similarly, we can define $S_R$ when replacing $\widehat{\mathbb{K}}_0$ in \eqref{equ:matU} and \eqref{equ:matL} using quadrature points with $\widetilde{\mathbb{K}}_0$ using random points. Figures~\ref{fig:arccosin K0} and
\ref{fig:arccosin K1} show the plots of the number of quadrature points on $\bS^3$ agains the scores $S_M$ for both quadrature measure and random measures for arccosine kernels $K_0$ and $K_1$. It can be seen that the quadrature score is better the score from random points for the smoother kernel $K_1$. We also compute the scores for a very smooth kernel $K_g$ defined in \eqref{defLKg}. Since all the scores $S_M$ are zeros, we listed the results in Table~\ref{tab:spectral score G}. The errors are computed against a reference kernel $\hat{K}_g$ computed with a quadrature of degree $n=60$ on $\bS^3$. 

\begin{table}[h]
\centering
\begin{tabular}{|l|l|l|l|l|l|l|}
\hline
&\multicolumn{2}{|c|}{$\Delta=\lambda=10^{-2}$}
&\multicolumn{2}{|c|}{$\Delta=\lambda=10^{-3}$}
&\multicolumn{2}{|c|}{$\Delta=\lambda=10^{-6}$}\\ \hline
$n$ &  $S_M$ & $S_R$ & $S_M$ & $S_R$ & $S$ & $S_R$ \\ \hline 
10 & 0  &  15.4615 &0  &  20.0825 & 0 & 25.3715 \\
20 & 0  &  4.2110  &0  &  7.5541  & 0 &  8.7889 \\
30 & 0  &  1.1911  &0  &  4.8493  & 0 &  5.0043  \\
40 & 0  &  0.0690  &0  &  2.3707  & 0 &  2.9237  \\
50 & 0  &  0.0321  &0  &  1.8609  & 0 &  1.8593 \\
60 & 0  &  0.0150 &  0 & 1.4455  &0  & 1.8462    \\
\hline
\end{tabular}
\caption{Spectral scores $S_M$ using kernel $K_g$ using $M=1000$ data points. The values $S_R$ are for random points.}
\label{tab:spectral score G}
\end{table}

\begin{figure}
\centering
\includegraphics[width=0.8\textwidth]{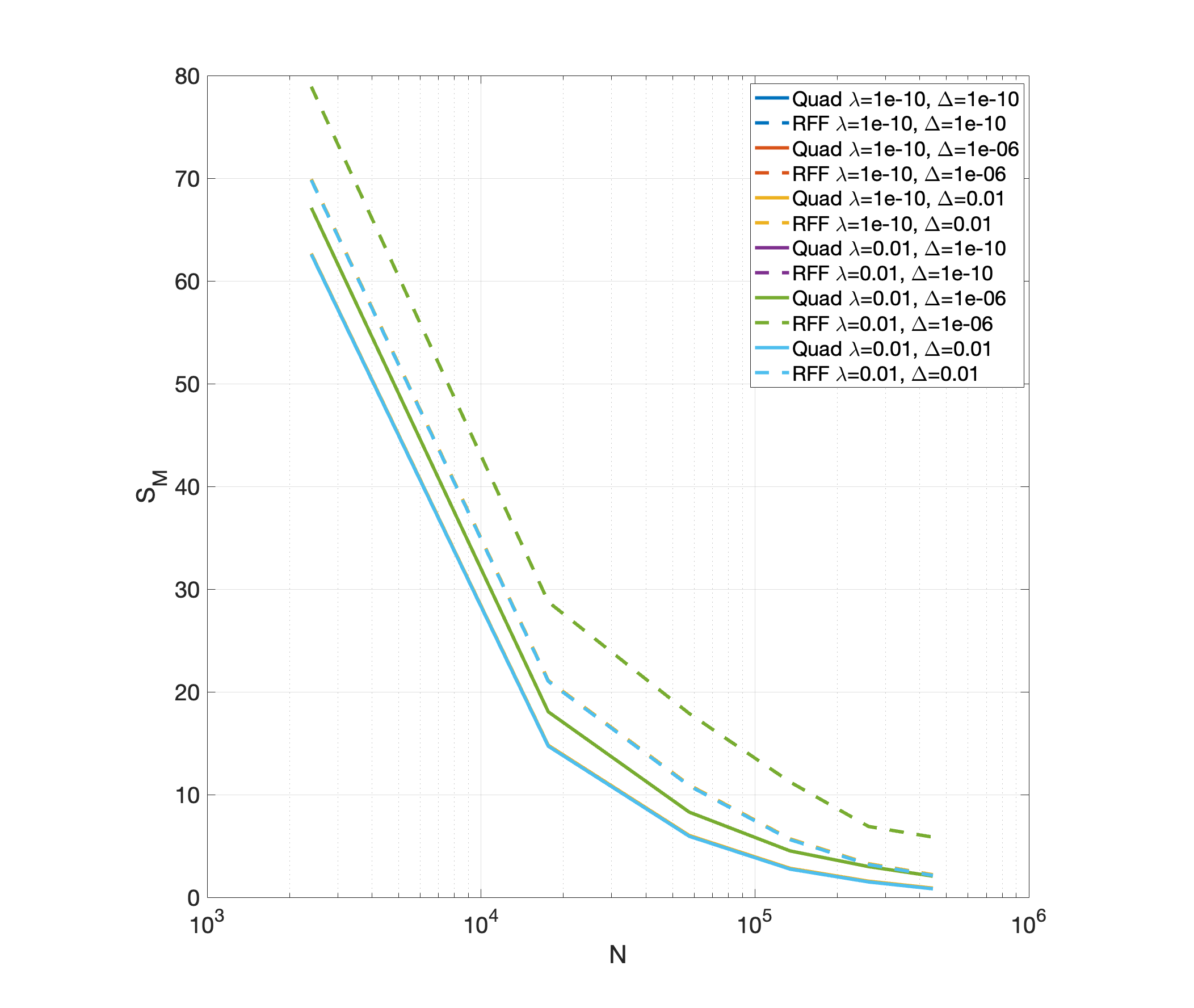}
\caption{Spectral scores for kernel built on $M=1000$ points sampled from arccosine kernel $K_0$ in $\mathbb{S}^3$.  Compares quadrature features vs random Fourier features.  }\label{fig:arccosin K0}
\end{figure}
\begin{figure}
\centering
\includegraphics[width=0.8\textwidth]{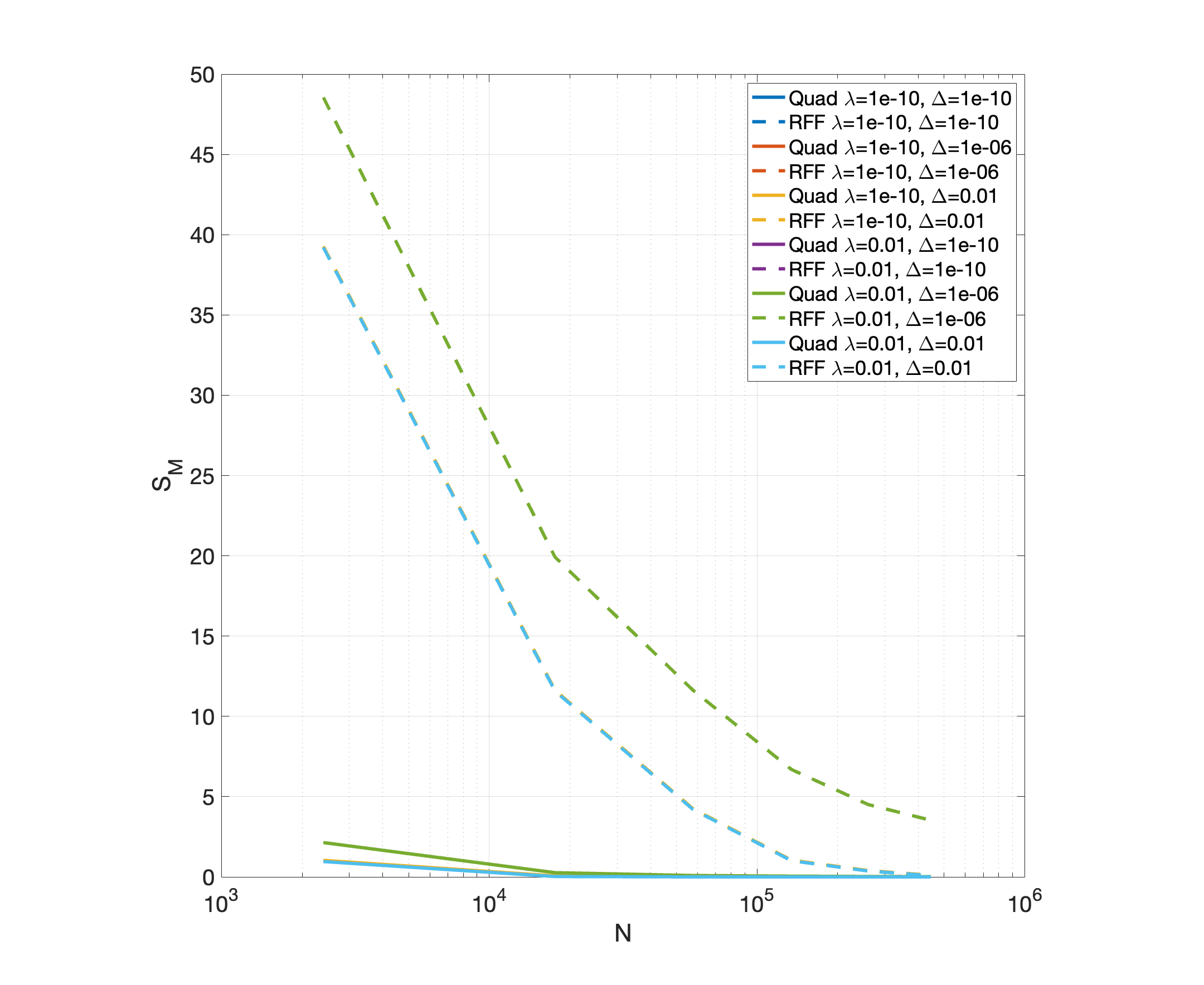}
\caption{Spectral scores for kernel built on $M=1000$ points sampled from arccosine kernel $K_1$ in $\mathbb{S}^3$.  Compares quadrature features vs random Fourier features.  }\label{fig:arccosin K1}
\end{figure}

\bhag{Conclusion}

In kernel based machine learning, one needs to consider a matrix of the form $[K(x_i,x_j)]_{i,j=1}^M$. This requires a storage of $M^2$ and the time complexity of the linear algebra operations may be as high as $\mathcal{O}(M^3)$.
The kernel trick to avoid this high complexity is to express $K$ in terms of an explicit kernel map, for example, random Fourier features for the Gaussian kernel. 
In this paper, we demonstrate that a systematic use of well designed quadrature formulas can yield a substantial improvement in the estimation of the eigenvalues of the kernel $K$ compared to random features. 
The theory is illustrated using numerical examples in the case of the Gaussian kernel using Gauss-Hermite quadrature formulas and in the case of some spherical neural tangent kernels using a tensor product quadrature formula.

\bibliographystyle{abbrv}
\bibliography{refs-pap}
\end{document}